\documentclass{amsart}
\usepackage{amsmath}
\usepackage{amsfonts}
\usepackage{amssymb}
\usepackage{ascmac}
\usepackage{layout}
\usepackage{amsthm}

\makeatletter
\@namedef{subjclassname@2020}{%
  \textup{2020} Mathematics Subject Classification}
\makeatother

\title{On the centeredness of saturated ideals}
\author{Kenta Tsukuura}

\address{University of Tsukuba, Tsukuba, 305-8571, Japan}
\email{tukuura@math.tsukuba.ac.jp} 
\thanks{This research was supported by Grant-in-Aid for JSPS Research Fellow Number 20J21103.}


\newcommand{\force}{\Vdash}

\theoremstyle{plain}
\newtheorem{thm}{Theorem}[section]
\newtheorem{lem}[thm]{Lemma}
\newtheorem{clam}[thm]{Claim}

\newtheorem{ques}[thm]{Question}
\newtheorem{rema}[thm]{Remark}

\begin{document}
\maketitle
\begin{abstract}
 We show that Kunen's saturated ideal over $\aleph_1$ is not centered. We also evaluate the extent of saturation of Laver's saturated ideal in terms of $(\kappa,\lambda,<\nu)$-saturation.
\end{abstract}

\section{Introduction}
In \cite{MR495118}, Kunen established  
\begin{thm}[Kunen~\cite{MR495118}]\label{kunentheorem}
 Suppose that $j:V \to M$ is a huge embedding with critical point $\kappa$. Then there is a poset $P$ such that $P\ast\dot{S}(\kappa,j(\kappa))$ forces that $\aleph_1$ carries a saturated ideal.
\end{thm}
This theorem has been improved in some ways. One is due to Foreman and Laver \cite{MR925267}. They established 
\begin{thm}[Foreman--Laver~\cite{MR925267}]
 Suppose that $j:V \to M$ is a huge embedding with critical point $\kappa$. Then there is a poset $P$ such that $P\ast\dot{R}(\kappa,j(\kappa))$ forces that $\aleph_1$ carries a centered ideal.
\end{thm}
Centeredness is one of the strengthenings of saturation. See Section 2 for the definition of centered ideal. Foreman and Laver introduced the poset $R(\kappa,\lambda)$ to obtain the centeredness, while Kunen used the Silver collapse $S(\kappa,\lambda)$. In their paper, it is claimed without proof that the ideal in Theorem \ref{kunentheorem} is not centered. We aim to give a proof of this claim in greater generality. Indeed, we will show that
 
\begin{thm}\label{main}
 Suppose that $j:V \to M$ is a huge embedding with critical point $\kappa$ and $f:\kappa \to \mathrm{Reg}\cap \kappa$ satisfies $j(f)(\kappa) \geq \kappa$. For regular cardinals $\mu < \kappa \leq \lambda = j(f)(\kappa) < j(\kappa)$, there is a $P$ such that $P\ast\dot{S}(\lambda,j(\kappa))$ forces $\mu^{+} = \kappa$ and $\lambda^{+} = j(\kappa)$ and $\mathcal{P}_{\kappa}(\lambda)$ carries a saturated ideal that is not centered. 
\end{thm}
We can regard an ideal over $\mathcal{P}_{\kappa}\kappa$ as an ideal over $\kappa$. If we put $f = \mathrm{id}$, $\lambda = \kappa$ and $\mu = \aleph_0$, then $P$ and the ideal in Theorem \ref{main} are the same as those in Theorem \ref{kunentheorem}. 

We also study Laver's saturated ideal. Laver introduced Laver collapse $L(\kappa,\lambda)$ to get a model in which $\aleph_1$ carries a strongly saturated ideal. He established 
\begin{thm}[Laver~\cite{MR673792}]\label{lavertheorem}
 Suppose that $j:V \to M$ is a huge embedding with critical point $\kappa$. Then there is a poset $P$ such that $P\ast\dot{L}(\kappa,j(\kappa))$ forces that $\aleph_1$ carries a saturated ideal.
\end{thm}

We don't know whether Laver's saturated ideal is centered or not. But we study the extent of saturation of this ideal in terms of $(\kappa,\lambda,<\nu)$-saturation.

\begin{thm}\label{main2}
 Suppose that $j$ is a huge embedding with critical point $\kappa$ and $f:\kappa \to \mathrm{Reg}\cap \kappa$ satisfies $j(f)(\kappa) \geq \kappa$. For regular cardinals $\mu < \kappa \leq \lambda = j(f)(\kappa) < j(\kappa)$, there is a $P$ such that $P\ast\dot{L}(\lambda,j(\kappa))$ forces that $\mu^{+} = \kappa$, $\lambda^{+} = j(\kappa)$ and $\mathcal{P}_{\kappa}(\lambda)$ carries a strongly saturated ideal ${I}$, that is, ${I}$ is $(j(\kappa),j(\kappa),<\lambda)$-saturated. And ${I}$ is $(j(\kappa),\lambda,\lambda)$-saturated but not $(j(\kappa),j(\kappa),\lambda)$-saturated.
\end{thm}

 The structure of this paper is as follows: In Section 2, we recall the basic facts of forcing. We also recall Silver collapses. Section 3 is devoted to the proof of Theorem \ref{main}. In Section 4, we recall Laver collapse and we give the proof of Theorem \ref{main2}.

\section{Preliminaries}
In this section, we recall some definitions. We use \cite{MR1994835} as a reference for set theory in general. 

Our notation is standard. We use $\kappa,\lambda,\mu$ to denote a regular cardinal unless otherwise stated. We also use $\nu$ to denote a cardinal, possibly finite, unless otherwise stated. We write $[\kappa,\lambda)$ for the set of all ordinals between $\kappa$ and $\lambda$. By $\mathrm{Reg}$, we mean the set of all regular cardinals. For $\kappa < \lambda$, $E_{\geq \kappa}^{\lambda}$ and $E_{<\kappa}^{\lambda}$ denote the set of all ordinals below $\lambda$ of cofinality $\geq \kappa$ and $< \kappa$, respectively. 

Throughout this paper, we identify a poset $P$ with its separative quotient. Thus, $p \leq q \leftrightarrow \forall r \leq p (r || q) \leftrightarrow p \force q \in \dot{G}$, where $\dot{G}$ is the canonical name of $(V,P)$-generic filter. 

We say that $P$ is $\kappa$-centered if there is a sequence of centered subsets $\langle C_{\alpha} \mid \alpha < \kappa \rangle$ with $P = \bigcup_{\alpha<\kappa}C_{\alpha}$. A centered subset is $C \subseteq P$ such that every $X\in [C]^{<\omega}$ has a lower bound in $P$. We call such sequence a centering family of $P$. It is easy to see that the $\kappa$-centeredness implies the $\kappa^{+}$-c.c. 

We say that $P$ is well-met if $\prod X \in P$ for all $X \subseteq P$ with $X$ has a lower bound. If $P$ is well-met, the $\kappa$-centeredness of $P$ is equivalent to the existence of $\kappa$-many filters that cover $P$. Note that every poset that we will deal with in this paper is well-met.

For a filter $F\subseteq Q$, by $Q / F$, we mean the subset $\{q \in Q \mid \forall p \in F(p \parallel q)\}$ ordered by $\leq_Q$. For a given complete embedding $\tau:P \to Q$, $P \ast (Q / \tau ``\dot{G})$ is forcing equivalent with $Q$. We also write $Q / \dot{G}$ for $Q / \tau ``\dot{G}$ if ${\tau}$ is obvious from context. If the inclusion mapping $P \to Q$ is complete, we say that $P$ is a complete suborder of $Q$, denoted by $P \mathrel{\lessdot} Q$. 

In this paper, by ideal, we mean normal and fine ideal. For an ideal $I$ over $\mathcal{P}_{\kappa}\lambda$, $\mathcal{P}(\mathcal{P}_{\kappa,\lambda})/ I$ is $\mathcal{P}(\kappa^{+}) \setminus I$ ordered by $A \leq B \leftrightarrow A \setminus B \in I$. We say that $I$ is saturated and centered if $\mathcal{P}(\kappa^{+}) / I$ has the $\kappa^{++}$-c.c. and $\mathcal{P}(\kappa^{+})/I$ is $\kappa^{+}$-centered, respectively. 

The ideal $I$ over $\kappa^{+}$ is $(\alpha,\beta,<\gamma)$-saturated if $\mathcal{P}(\kappa^{+})/ I$ has the $(\alpha,\beta,<\gamma)$-c.c. in the sense of Laver~\cite{MR673792}. Whenever $I$ is $(\lambda^{+},\lambda^{+},<\lambda)$-saturated, we say that $I$ is strongly saturated. 

For a stationary subset $S \subseteq \lambda$, $P$ is $S$-layered if there is a sequence $\langle P_\alpha \mid \alpha < \lambda\rangle$ of complete suborders of $P$ such that $|P_{\alpha}| < \lambda$ and there is a club $C \subseteq \lambda$ such that $P_\alpha = \bigcup_{\beta < \alpha}P_{\beta}$ for all $\alpha \in S \cap C$. We say that $I$ is layered and if $\mathcal{P}(\mathcal{P}_{\kappa}\lambda)/ I$ is $S$-layered for some stationary $S \subseteq E^{\lambda^{+}}_{\lambda}$.

For cardinals $\kappa < \lambda$ ($\lambda$ is not necessary regular), Silver collapse $S(\kappa,\lambda)$ is the set of all $p$ with the following properties:
\begin{itemize}
 \item $p \in \prod_{\gamma \in [\kappa^{+},\lambda) \cap \mathrm{Reg}}^{\leq \kappa}{^{<\kappa}\gamma}$. 
 \item There is a $\xi < \kappa$ with $\forall \gamma \in \mathrm{dom}(p)(\mathrm{dom}(p(\gamma)) \subseteq \xi)$.
\end{itemize}
$S(\kappa,\lambda)$ is ordered by reverse inclusion. The following properties are well known.
\begin{lem}
 \begin{enumerate}
  \item $S(\kappa,\lambda)$ is $\kappa$-closed.
  \item If $\lambda$ is inaccessible, then $S(\kappa,\lambda)$ has the $\lambda$-c.c. Therefore, $S(\kappa,\lambda) \force \kappa^{+} = \lambda$. 
  \item If $\mu<\lambda$ then $S(\kappa,\mu) \mathrel{\lessdot} S(\kappa,\lambda)$.
 \end{enumerate}
\end{lem}
The following lemma will be used in a proof of Theorem \ref{main}.
\begin{lem}\label{dispoint}
\begin{enumerate}
 \item If $\mathrm{cf}(\delta)\leq\kappa$ then $S(\kappa,\delta)$ has an anti-chain of size $\delta^{+}$. 
 \item If $\mathrm{cf}(\delta)>\kappa$ and $\delta^{\kappa} = \delta$ then $S(\kappa,\delta)$ forces $(\delta^{+})^{V} \geq \kappa^{+}$. 
\end{enumerate}
\end{lem}
\begin{proof}
 (2) follows by a standard cardinal arithmetic. Indeed, The assumption shows $|S(\kappa,\delta)| = \delta$. Therefore $S(\kappa,\delta)$ has the $\delta^{+}$-c.c. We only check about (1). By  $\mathrm{cf}(\delta) < \kappa$, we can fix an increasing sequence of regular cardinals $\langle \delta_{i} \mid i < \mathrm{cf}(\delta) \rangle$ which converges to $\delta$. For $f \in \prod_{i < \mathrm{cf}(\delta)} \delta_{i}$, define $p_f \in S(\kappa,\delta)$ by $\mathrm{dom}(p_f) = \{\delta_{i} \mid i < \mathrm{cf}(\delta)\}$ and $p_f(\delta_{i}) = \{\langle 0,f(i)\rangle\}$. It is easy to see that $f \not= g$ implies $p_f \perp p_g$. Therefore $\{p_f \mid f \in \prod_{i < \mathrm{cf}}\delta_i\}$ is an anti-chain of size $(\delta^{\mathrm{cf}\delta}) \geq \delta^{+}$, as desired.

\end{proof}
\begin{rema}\label{levyremark}
 A similar proof shows the following analogue of Lemma \ref{dispoint}.
\begin{enumerate}
 \item If $\mathrm{cf}(\delta)<\kappa$ then $\mathrm{Coll}(\kappa,<\delta)$ has an anti-chain of size $\delta^{+}$. 
 \item If $\mathrm{cf}(\delta)\geq\kappa$ and $\delta^{<\kappa} = \delta$ then $\mathrm{Coll}(\kappa,<\delta)$ forces $(\delta^{+})^{V} \geq \kappa^{+}$. 
\end{enumerate}
\end{rema}

\section{Proof of Theorem \ref{main}}
Let $j:V \to M$ be a huge embedding with critical point $\kappa$. Fix $\mu < \kappa$. We also fix $f:\kappa \to \kappa \cap \mathrm{Reg}$ with $\kappa \leq j(f)(\kappa) = \lambda$. We may assume that $f(\alpha) \geq \alpha$ for all $\alpha$. 

Let $\langle{P_{\alpha} \mid \alpha \leq \kappa\rangle}$ be the $<\mu$-support iteration such that
\begin{itemize}
 \item $P_{0} = S(\mu,\kappa)$.
 \item $P_{\alpha + 1} = \begin{cases}P_{\alpha} \ast S^{P_\alpha \cap V_{\alpha}}(f(\alpha),\kappa) & \alpha\text{ is good}\\
			  P_{\alpha} & \text{otherwise}
			 \end{cases}$.
\end{itemize}
Here, we say that $\alpha$ is good if $P_{\alpha} \cap V_{\alpha} \mathrel{\lessdot} P_{\alpha}$ has the $\alpha$-c.c., $\alpha$ is inaccessible, and $\alpha \geq \mu$. The set $P_{\alpha} \ast S^{P_\alpha \cap V_{\alpha}}(\alpha,\kappa)$ is the set of all $\langle p,\dot{q}\rangle$ such that $p \in P_\alpha$ and $\dot{q}$ is a $P_{\alpha} \cap V_\alpha$-name for an element of $S^{P_\alpha \cap V_{\alpha}}(\alpha,\kappa)$. 

Define $P = P_{\kappa}$. For every $p \in P$ and good $\alpha$, we may assume that $p(\alpha)$ is $P_\alpha \cap V_{\alpha}$-name. This $P$ is called an universal collapse.  $P$ has the following properties:
\begin{lem}\label{basicpropofunivcoll}
 \begin{enumerate}
  \item $P$ is $\mu$-directed closed and has the $(\kappa,\kappa,<\mu)$-c.c.
  \item $P \subseteq V_{\kappa}$ and $P \force \mu^{+} = \kappa$.
  \item $\kappa$ is good for $j(P)$. In particular, $j(P)_{\kappa} \cap V_{\kappa} = P \mathrel{\lessdot} j(P)_{\kappa}$.
  \item There is a complete embedding $\tau:P \ast \dot{S}(\lambda,j(\kappa)) \to j(P)_{\kappa +1}\lessdot j(P)$ such that $\tau(p,\emptyset) = p$ for all $p \in P$. 
 \end{enumerate}
\end{lem}
Note that $j(P)$ has the $j(\kappa)$-c.c. by the hugeness of $j$. By Lemma \ref{basicpropofunivcoll} (2), in the extension by $P \ast \dot{S}(\lambda,j(\kappa))$, $\mu^{+} = \kappa$, $\lambda^{+} = j(\kappa)$, and every cardinal between $\kappa$ and $\lambda$ are preserved.

Kunen proved Theorem \ref{Kunenideal} in the case of $\kappa = \lambda$ and $\mu = \omega$. Moreover,
\begin{thm}\label{Kunenideal}
 $P \ast \dot{S}(\lambda,j(\kappa))$ forces that ${\mathcal{P}}_{\kappa}\lambda$ carries a saturated ideal $\dot{I}$.
\end{thm}
\begin{proof}
 The same proof as in \cite[Section 7.7]{MR2768692} works.
\end{proof}
The ideal which we call ``Kunen's saturated ideal'' is this $\dot{I}$. Studying the saturation of $\dot{I}$ will be reduced to that of some quotient forcing. Indeed, Theorem \ref{main} follows from Lemmas \ref{fms} and \ref{rephrasedmain}. We only give a proof of Lemma \ref{rephrasedmain}. In \cite{MR942519}, Foreman, Magidor and Shelah proved Lemma \ref{fms} in the case of $\kappa = \lambda$.
\begin{lem}\label{fms}
 $P \ast \dot{S}(\lambda,j(\kappa))$ forces ${\mathcal{P}}({\mathcal{P}}_{\kappa}\lambda) / \dot{I} \simeq j(P) / \dot{G} \ast \dot{H}$. Here, $\dot{G} \ast \dot{H}$ is the canonical name for generic filter.
\end{lem}
\begin{proof}
 The same proof as in \cite[Claim 7]{MR942519} works.
\end{proof}

\begin{lem}\label{rephrasedmain}
 $P \ast \dot{S}(\lambda,j(\kappa))$ forces $j(P) / \dot{G} \ast \dot{H}$ is not $\lambda$-centered. 
\end{lem}

\begin{proof}

Note that $\{\alpha < j(\kappa) \mid \alpha$ is good$\}$ is unbounded in $j(\kappa)$. We fix a good $\alpha > \kappa$. It is enough to prove that $j(P)_{\alpha + 1} / \dot{G} \ast \dot{H}$ is not $\lambda$-centered in the extension. 

We show by contradiction. Suppose the existence of a centering family $\langle \dot{C}_{\xi} \mid \xi < \lambda \rangle$ of $j(P)_{\alpha + 1} / \dot{G} \ast \dot{H}$ is forced by some condition. We may assume that each $\dot{C}_{\xi}$ is forced to be a filter. To simplify notation, we assume $P \ast \dot{S}(\lambda,j(\kappa))$ forces the existence of such a 
centering family. 
By the $\kappa$-c.c. of $P$, for every $\langle{p,\dot{q}}\rangle \in P \ast \dot{S}(\lambda,j(\kappa))$, $P \force \dot{q} \in \dot{S}(\lambda,\beta)$ for some $\beta < j(\kappa)$. For each $q \in j(P)_{\alpha+1}$, let $\rho(q)$ be defined by the following way:

For $\xi < \lambda$, let $\mathcal{A}_{q}^{\xi} \subseteq P \ast \dot{S}(\lambda,j(\kappa))$ be a maximal anti-chain such that, for every $r \in \mathcal{A}_{q}^{\xi}$, $r$ decides $q \in \dot{C}_{\xi}$. $\rho(q)$ is the least ordinal $\beta<j(\kappa)$ such that $P \force \dot{q} \in \dot{S}(\lambda,\beta)$ for every $\langle p,\dot{q}\rangle \in \bigcup_{\xi}\mathcal{A}_{q}^{\xi}$. 

 We put $Q = j(P)_{\alpha} \cap V_{\alpha}$. Let $C \subseteq j(\kappa)$ be a club generated by $\beta \mapsto \sup\{\rho(q) \mid q \in Q \ast {S}^{Q}(\alpha,\beta) \}$.  Since $j(\kappa)$ is inaccessible, we can find a strong limit cardinal $\delta \in C \cap E^{j(\kappa)}_{> \lambda} \cap E^{j(\kappa)}_{\leq\alpha} \setminus (\alpha + 1)$. 

 By the $\kappa$-c.c. of $P$ and Lemma \ref{dispoint} (2), $P \ast \dot{S}(\lambda,\delta) \force (\delta^{+})^{V} \geq \lambda^{+}$. We will discuss in the extension by $P \ast \dot{S}(\lambda,\delta)$. Let $\dot{G} \ast \dot{H}_{\delta}$ be the canonical $P \ast \dot{S}(\lambda,\delta)$-name for a generic filter.

 By Lemma \ref{dispoint} (1), $P \ast \dot{S}(\lambda,\delta) \force S^{V}(\alpha,\delta)$ has an anti-chain of size $(\delta^{+})^{V} \geq \lambda^{+}$. By Lemma \ref{basicpropofunivcoll} (4), this $S^{V}(\alpha,\delta)$-anti-chain defines an anti-chain in $Q \ast S^Q(\alpha,\delta) / \dot{G} \ast \dot{H}_{\delta}$ in the extension. Thus, $Q \ast S^Q(\alpha,\delta) / \dot{G} \ast \dot{H}_{\delta}$ is forced not to have the $\lambda^{+}$-c.c., and thus not to be $\lambda$-centered. 

On the other hand, a name of a centering family of $j(P)_{\alpha + 1} / \dot{G} \ast \dot{H}$ defines a centering family of $Q \ast S^{Q}(\alpha,\delta) / \dot{G} \ast \dot{H}_{\delta}$ in the extension as follows.
\begin{clam}\label{claim1}
 $P \ast \dot{S}(\lambda,\delta)$ forces that $Q \ast {S}^{Q}(\alpha,\delta) / \dot{G} \ast \dot{H}_{\delta}$ is $\lambda$-centered.
\end{clam}

\begin{proof}[Proof of Claim]
 For every $q \in \bigcup_{\beta<\delta} Q \ast {S}^{Q}(\alpha,\beta)$, by $\rho(q) < \delta$, the statement $q \in \dot{C}_{\xi}$ has been decided by $P \ast \dot{S}(\lambda,\delta)$ for all $\xi < \lambda$. Let $G \ast H$ be an arbitrary $(V,P \ast \dot{S}(\lambda,j(\kappa)))$-generic filter. Note that $G \ast H_\delta = G \ast H \cap (P \ast \dot{S}(\lambda,\delta))$ is $(V,P \ast \dot{S}(\lambda,\delta))$-generic. Let $D_{\xi}$ be defined by 
\begin{center}
 $\langle p,\dot{q}\rangle \in D_\xi$ if and only if $\langle p,\dot{q} \upharpoonright \beta \rangle \in \dot{C}_\xi$ forced by $G \ast H_\delta$ for every $\beta<\delta$. 
\end{center}
 It is easy to see that $D_{\xi}$ is a filter over $Q\ast {S}^{Q}(\alpha,\delta) / {G} \ast {H}_{\delta}$. We claim that $\{D_\xi\mid \xi < \lambda\}$ covers $Q \ast {S}^{Q}(\alpha,\delta) / G \ast H_\delta$. For each $\langle p,\dot{q}\rangle \in Q\ast \dot{S}(\alpha,\delta) / {G} \ast {H}_{\delta}$, in $V[G][H]$, there is a $\xi$ such that $\langle{p,\dot{q}}\rangle \in \dot{C}_\xi^{G \ast H}$. Then $\langle{p,\dot{q}\upharpoonright \beta} \rangle \in \dot{C}_{\xi}^{G \ast H}$ for every $\beta < \delta$, since $\dot{C}_{\xi}^{G \ast H}$ is a filter. In particular, $\langle{p,\dot{q}\upharpoonright \beta} \rangle \in \dot{C}_{\xi}$ is forced by $G\ast H_\delta$ for every $\beta < \delta$. By the definition of $D_\xi$, $\langle {p,\dot{q}}\rangle \in D_{\xi}$ in $V[G][H_{\delta}]$, as desired.
\end{proof}
This claim shows the contradiction. The proof is completed.
\end{proof}
Let us give a proof of Theorem \ref{main}.
\begin{proof}[Proof of Theorem \ref{main}]
 By Lemmas \ref{fms} and \ref{rephrasedmain}, we have $\force \mathcal{P}({\mathcal{P}}_{\kappa}\lambda) / \dot{I}$ is not $\lambda$-centered. Thus, $\dot{I}$ is not centered ideal in the extension.
\end{proof}

Lastly, we list other saturation properties of $\dot{I}$. 
\begin{thm}\label{extentofkunen}$P \ast \dot{S}(\lambda,j(\kappa))$ forces that
 \begin{enumerate}
  \item $\dot{I}$ is $(j(\kappa),j(\kappa),<\mu)$-saturated.
  \item $\dot{I}$ is not $(j(\kappa),\mu,\mu)$-saturated. In particular, $\dot{I}$ is not strongly saturated.
  \item $\dot{I}$ is layered.
  \item $\dot{I}$ is not centered. 
 \end{enumerate}
\end{thm}
\begin{proof}
 For (1) and (2), we refer to \cite{preprint}. (3) has been proven in \cite{MR942519}. (4) follows by Lemma \ref{rephrasedmain}. 
\end{proof}

\begin{rema}
 Kunen's theorem can be improved by Magidor's trick which appeared in \cite{MR526312}. This shows that an almost huge cardinal is enough to show Theorem \ref{kunentheorem}. Then we can use Levy collapse instead of Silver collapse. Remark \ref{levyremark} enables us to show the same result with Theorem \ref{main} for Levy collapses. On the other hand, to obtain layeredness, we need an almost huge embedding $j:V \to M$ with $j(\kappa)$ is Mahlo. If $j(\kappa)$ is not Mahlo, the ideal $\dot{I}$ is forced to be not layered. For details, we refer to \cite{preprint}.
\end{rema}

\section{Proof of Theorem \ref{main2}}
In this section, we give a proof of Theorem \ref{main2}. 

First, we recall the definition and basic properties of Laver collapse $L(\kappa,\lambda)$. $L(\kappa,\lambda)$ is the set of all $p$ such that 
\begin{itemize}
 \item $p \in \prod_{\gamma \in [\kappa^{+},\lambda) \cap \mathrm{Reg}}^{<\lambda}{^{<\kappa}\gamma}$. 
 \item There is a $\xi < \kappa$ with $\forall \gamma \in \mathrm{dom}(p)(\mathrm{dom}(p(\gamma)) \subseteq \xi)$.
 \item $\mathrm{dom}(p) \subseteq \lambda$ is Easton subset. That is, $\forall \alpha \in \mathrm{Reg} (\sup (\mathrm{dom}(p) \cap \alpha) < \alpha)$. 
\end{itemize}
$L(\kappa,\lambda)$ is ordered by reverse inclusion. It is easy to see that 
\begin{lem}
 \begin{enumerate}
  \item $L(\kappa,\lambda)$ is $\kappa$-closed.
  \item If $\lambda$ is Mahlo, then $L(\kappa,\lambda)$ has the $(\lambda,\lambda,<\mu)$-c.c. for all $\mu < \lambda$. Therefore, $L(\kappa,\lambda) \force \kappa^{+} = \lambda$. 
 \end{enumerate}
\end{lem}
\begin{proof}
 (1) follows by the standard argument. (2) follows by the usual $\Delta$-system argument.
\end{proof}

For $\kappa,\lambda,\mu,j,f$ in the assumption of Theorem \ref{main2}, let us define $P$. Let $\langle P_{\alpha} \mid \alpha \leq \kappa \rangle$ be the Easton support iteration of such that
\begin{itemize}
 \item $P_{0} = L(\mu,\kappa)$.
 \item $P_{\alpha + 1} = \begin{cases}P_{\alpha} \ast L^{P_\alpha \cap V_{\alpha}}(f(\alpha),\kappa) & \alpha\text{ is good}\\
			  P_{\alpha} & \text{otherwise}
			 \end{cases}$.
\end{itemize}
Goodness of $\alpha$ is defined as in Section 3. Define $P = P_\kappa$. Then 
\begin{lem}\label{basicpropofunivcoll2}
 \begin{enumerate}
  \item $P$ is $\mu$-directed closed and has the $(\kappa,\kappa,<\nu)$-c.c for all $\nu < \kappa$.
  \item $P \subseteq V_{\kappa}$ and $P \force \mu^{+} = \kappa$.
  \item $\kappa$ is good for $j(P)$. In particular, $j(P)_{\kappa} \cap V_{\kappa} = P \mathrel{\lessdot} j(P)_{\kappa}$.
  \item There is a complete embedding $\tau:P \ast \dot{L}(\lambda,j(\kappa)) \to j(P)_{\kappa + 1} \lessdot j(P)$ such that $\tau(p,\emptyset) = p$ for all $p \in P$. 
 \end{enumerate}
\end{lem}
Laver proved the following in the case of $\kappa = \lambda$. But the same proof showed
\begin{thm}
 $P\ast \dot{L}(\lambda,j(\kappa))$ forces that ${\mathcal{P}}_{\kappa}\lambda$ carries a strongly saturated ideal $\dot{I}$. 
\end{thm}

We have an analogue of Lemma \ref{fms} for Laver's ideal.
\begin{lem}\label{dualitylaver} 
 $P \ast \dot{L}(\lambda,j(\kappa))$ forces $\mathcal{P}({\mathcal{P}}_{\kappa}\lambda) / \dot{I} \simeq j(P) / \dot{G} \ast \dot{H}$. Here, $\dot{G} \ast \dot{H}$ is the canonical name for generic filter.
\end{lem}\begin{proof}
 The same proof as in \cite[Claim 7]{MR942519} works.
\end{proof}

\begin{lem}\label{quotientsaturation}
 Suppose that $\tau:P \to Q$ is complete embedding between complete Boolean algebras and $Q$ has the $(\kappa,\mu,\mu)$-c.c. Then $P \force Q / \dot{G}$ has the $(\kappa,\mu,\mu)$-c.c.
\end{lem}
\begin{proof}
 We may assume that $P$ and $Q$ are complete Boolean algebras.
 Let $p \force \{\dot{q}_{\alpha} \mid \alpha < \kappa\} \subseteq Q / \dot{G}$ be arbitrary. For each $\alpha < \kappa$, there are $p_{\alpha} \leq p$ and $q_{\alpha} \in Q$ such that $p_{\alpha} \force \dot{q}_{\alpha} = q_{\alpha}$. By the $(\kappa,\mu,\mu)$-c.c. of $Q$, there is a $Z \in [\kappa]^{\mu}$ such that $\prod_{\alpha \in Z} \tau(p_\alpha) \cdot q_{\alpha} \not= 0$. It is easy to see that 
\begin{center}
 $\prod_{\alpha \in Z} \tau(p_\alpha) \cdot q_{\alpha} = \prod_{\alpha \in Z}\tau(p_{\alpha}) \cdot \prod_{\alpha \in Z}q_{\alpha} = \tau(\prod_{\alpha \in Z} p_{\alpha}) \cdot \prod_{\alpha \in Z}q_{\alpha}$.
\end{center}
Let $r$ be a reduct of $\tau(\prod_{\alpha \in Z} p_{\alpha}) \cdot \prod_{\alpha \in Z}q_{\alpha}$. Then $r \leq \prod_{\alpha \in Z}p_{\alpha}\leq p$ and this forces that $\prod_{\alpha \in Z}q_{\alpha} \in Q / \dot{G}$ is a lower bound of $\{\dot{q}_{\alpha} \mid \alpha \in Z\}$, as desired. 
\end{proof}

\begin{lem}\label{mainlemma2}
 $P \ast \dot{L}(\lambda,j(\kappa))$ forces that $j(P) / \dot{G} \ast \dot{H}$ does not have the $(j(\kappa),j(\kappa),\lambda)$-c.c.
\end{lem}
\begin{proof}
 Let us show $P \ast \dot{L}(\lambda,j(\kappa)) \force j(P) / \dot{G} \ast \dot{H}$ does not has the $(j(\kappa),j(\kappa),\lambda)$-c.c. Let $\{q_{\alpha} \in j(P) \mid \alpha > \kappa + 1\}$ be an arbitrary such that $\mathrm{supp}(q_{\alpha}) = \{\alpha\}$ for all $\alpha$. Note that $\langle \emptyset,\emptyset \rangle \in P \ast \dot{L}(\lambda,j(\kappa))$. For every $p \in P\ast \dot{L}(\lambda,j(\kappa))$, $\tau(p) \in j(P)_{\kappa + 1}$ by Lemma \ref{basicpropofunivcoll2}(4). Therefore $\tau(p)$ meets with $q_{\alpha}$. We have $\force \{q_{\alpha} \mid \alpha > \kappa + 1\} \subseteq j(P) / \dot{G} \ast \dot{H}$. 
 
We fix an arbitrary $\dot{A}$ with $\force \dot{A} \in [j(\kappa)]^{j(\kappa)} $. Let us find an $\dot{x}$ such that $\force \dot{x} \in [\dot{A}]^{\lambda}$ and $\{q_{\alpha} \mid \alpha \in \dot{x} \}$ has a lower bound in $j(P) / \dot{G} \ast \dot{H}$. By the $j(\kappa)$-c.c. of $P \ast \dot{j}(\lambda,j(\kappa))$, there is a club $C \subseteq j(\kappa)$ such that $\force C \subseteq \mathrm{Lim}(\dot{A})$. Since $j(\kappa)$ is Mahlo, there is an inaccessible $\alpha \in C\setminus (\lambda + 1)$. Let $\dot{x}$ be a $P \ast \dot{L}(\lambda,j(\kappa))$-name for the set $\dot{A} \cap \alpha$. Since $\force \alpha \in \mathrm{Lim}(\dot{A})$, $\sup \dot{x} = \alpha$. Since $\lambda \leq \alpha < j(\kappa)$, $|\dot{x}| = |\alpha| = \lambda$ is forced by $P\ast \dot{L}(\lambda,j(\kappa))$. We claim that $\force \{q_\alpha \mid \alpha \in \dot{x}\}$ witnesses. Suppose otherwise, there are $p \in P \ast \dot{L}(\lambda,j(\kappa))$ and $r \in j(P)$ such that $p \force r \in j(P) / \dot{G}\ast\dot{H} \leq q_{\alpha}$ for all $\alpha \in \dot{x}$. Note that $\mathrm{supp}(r) \cap \alpha < \beta$ for some $\beta < \alpha$. By $q \force \sup \dot{x} = \alpha$, there are $q \leq p$ and $\gamma \in [\beta^{+},\alpha)$ such that $q \force \gamma \in \dot{x}$. Therefore $q \force r \leq q_{\gamma}$, and thus, $\{\gamma\} \subseteq \mathrm{supp}(r)\cap [\beta^{+},\alpha)$. This is a contradiction.
\end{proof}
Let us show 
\begin{proof}[Proof of Theorem \ref{main2}]
Note that $j(P)$ has the $(j(\kappa),j(\kappa),\mu)$-c.c. for all $\mu < j(\kappa)$. Therefore $j(P)$ has the $(j(\kappa),\lambda,\lambda)$-c.c. By Lemma \ref{quotientsaturation}, $j(P) / \dot{G} \ast \dot{H}$ is forced to have the $(j(\kappa),\lambda,\lambda)$-c.c. By Lemma \ref{mainlemma2}, $P \ast \dot{L}(\lambda,j(\kappa))$ forces that $\mathcal{P}({\mathcal{P}}_{\kappa}(\lambda)) / \dot{I}$ does not have the $(j(\kappa),j(\kappa),\lambda)$-c.c. 
 By Lemma \ref{dualitylaver}, $\dot{I}$ is forced to have the $(j(\kappa),\lambda,\lambda)$-c.c. but not the $(j(\kappa),j(\kappa),\lambda)$-c.c. 
\end{proof}

\begin{thm}
 $P \ast \dot{L}(\lambda,j(\kappa))$ forces that 
\begin{enumerate}
 \item $\dot{I}$ is $(j(\kappa),j(\kappa),<\lambda)$-saturated and $(j(\kappa),\lambda,\lambda)$-saturated.
 \item $\dot{I}$ is not $(j(\kappa),j(\kappa),\lambda)$-saturated.
 \item $\dot{I}$ is layered. 
\end{enumerate}
\end{thm}
\begin{proof}
 (1) and (2) follow from Lemma \ref{mainlemma2}. (3) follows by the proof in \cite{MR942519}. 
\end{proof}

Shioya improved Theorem \ref{lavertheorem} by using Easton collapse. Easton collapse $E(\kappa,\lambda)$ is the Easton support product $\prod^{E}_{\gamma \in [\kappa^{+},\lambda) \cap \mathrm{SR}}{^{<\kappa}\gamma}$. $\mathrm{SR}$ is the class of all cardinal $\gamma$ with $\gamma^{<\gamma} = \gamma$. He showed
\begin{thm}[Shioya~\cite{MR4159767}]\label{shioya}
Suppose that $j:V \to M$ is an almost-huge embedding with critical point $\kappa$ and $j(\kappa)$ is Mahlo. For regular cardinals $\mu < \kappa \leq \lambda < j(\kappa)$, $E(\mu,\kappa) \ast \dot{E}(\lambda,j(\kappa))$ forces that $\mu^{+} = \kappa$, $j(\kappa) = \lambda^{+}$ and $\mathcal{P}_{\kappa}\lambda$ carries a strongly saturated ideal.
\end{thm}
Let $\dot{I}$ be a $E(\mu,\kappa) \ast \dot{E}(\lambda,j(\kappa))$-name for the ideal in Theorem \ref{shioya}. The similar proof shows that 
\begin{thm}$E(\mu,\kappa) \ast \dot{E}(\lambda,j(\kappa))$ forces that
\begin{enumerate}
 \item $\dot{I}$ is $(j(\kappa),j(\kappa),<\lambda)$-saturated and $(j(\kappa),\lambda,\lambda)$-saturated.
 \item $\dot{I}$ is not $(j(\kappa),j(\kappa),\lambda)$-saturated.
 \item $\dot{I}$ is layered. 
\end{enumerate} 
\end{thm}

We conclude this paper with the following question.
\begin{ques}
 Does $P \ast \dot{L}(\lambda,j(\kappa))$ force that $\dot{I}$ is centered? What about $E(\mu,\kappa)\ast\dot{E}(\lambda,j(\kappa))$?
\end{ques}

\bibliographystyle{plain}
\bibliography{ref}

\end{document}